\documentclass[12pt]{amsart}
\usepackage{amsmath, amssymb}
\usepackage{xcolor, hyperref, comment}
\newtheorem{theorem}{\sc Theorem}[section]
\newtheorem{lemma}[theorem]{\sc Lemma}
\newtheorem{proposition}[theorem]{\sc Proposition}
\newtheorem{corollary}[theorem]{\sc Corollary}

\usepackage{abstract} % Permite maior personalização

\title[Variations of Baer's theorem]{New variations on the theme of Baer's theorem}

\author[M. Capasso]{Martina Capasso } 
\address{Martina Capasso: Dipartimento di Matematica e Applicazioni, Università di Napoli Federico II, Complesso Universitario Monte S. Angelo, Naples (Italy)}
\email{martina.capasso2@unina.it}

\author[L. Lancellotti]{Liliana Lancellotti} 
\address{Liliana Lancellotti: Dipartimento di Matematica e Applicazioni, Università di Napoli Federico II, Complesso Universitario Monte S. Angelo, Naples (Italy)}
\email{liliana.lancellotti@unina.it}

\author[P. Shumyatsky]{Pavel Shumyatsky} 
\address{Pavel Shumyatsky: Department of Mathematics, University of Brasilia, Brasilia, DF, Brazil}
\email{pavel@unb.br}

\thanks{The first and the second authors are members of GNSAGA (INdAM) and  of AGTA – Advances in Group Theory and Applications (\href{www.advgrouptheory.com}{www.advgrouptheory.com}), and their work was supported by GNSAGA (INdAM). The work of the third author was supported by FAPDF and CNPq. 
The first and second authors wish to thank the University of Brasília for its kind hospitality during the visit in which this work was prepared.}
\keywords{Chernikov group; Rank; Upper and lower central series.}

\subjclass[2020]{20F14}

\begin{document}

\maketitle

\begin{abstract}
Let $\gamma_s(G)$ and $Z_s(G)$ denote the $s$-th terms of the lower and upper central series of a group $G$, respectively. A classical theorem by R. Baer states that if $Z_s(G)$ has finite index $n$ in $G$, then $\gamma_{s+1}(G)$ is also finite. In this paper, we prove that if $G$ is a generalized soluble group such that the quotient $\gamma_s(G)/(\gamma_s(G) \cap Z_t(G))$ has finite rank $r$ for some $s,t$, then the rank of $\gamma_{s+t}(G)$ is finite and $(r,s,t)$-bounded. Moreover, a corresponding result replacing 
the finite-rank assumption by the condition to be a Chernikov group of bounded size is also obtained. These results extend recent generalizations of the classical Baer's theorem.
\end{abstract}

\section{Introduction}
Let $G$ be a group, and let $\gamma_s(G)$ and $Z_s(G)$ denote the $s$-th terms of the lower and upper central series of $G$, respectively.
A classical theorem by R. Baer states that if $Z_s(G)$ has finite index $n$ in $G$, then $\gamma_{s+1}(G)$ is also finite.
Moreover, the order of $\gamma_{s+1}(G)$ can be bounded in terms of $n$ and $s$ (see the proof of 14.5.1 in \cite{ROB}). Throughout the paper, we use the expression “$(a,b,c, \dots)$-bounded" to mean that a quantity is finite and bounded by a certain number depending only on the parameters $a,b,c, \dots$. A stronger form of Baer’s theorem was obtained by G. Fernández-Alcober and M. Morigi, who proved that if $\gamma_s(G)/(\gamma_s(G) \cap Z_t(G))$ has finite order $n$, then $\gamma_{s+t}(G)$ is finite and its order is bounded by a function of $n,s$ and $t$ (see [\cite{AlcMor}, Theorem B], and the last part of \cite{AlcMor} for the quantitative version). \par
Recall that a group $G$ is said to have \emph{finite (Prüfer or special) rank} if there is a positive integer $r$ such that every finitely generated subgroup of $G$ can be generated by at most $r$ elements, and $r$ is the least integer with this property. A group $G$ is called \emph{generalized radical} if it has an ascending series
whose quotients are either locally nilpotent or locally finite. Accordingly, a group $G$ is \emph{locally generalized radical} if every finitely generated
subgroup of $G$ is generalized radical.
A rank analogue of Baer’s result states that if $G$ is a locally generalized radical group and $G/Z_s(G)$ has finite rank $r$, then $\gamma_{s+1}(G)$ has finite and $(r,s)$-bounded rank. This result was obtained by  N. Yu. Makarenko in \cite{Makarenko} when $G$ is finite, and later extended by L. A. Kurdachenko and J. Otal in \cite{KO2013} to the class of locally generalized radical groups. The case $s = 1$ had already been proved in \cite{pavel2013}. 
 \par
Along similar lines, A. R. Jamali and S. Zandi  proved a rank analogue of the aforementioned result of Fernández-Alcober and Morigi \cite{AlcMor}, in the case where $G$ is a finite $p$-group (see [\cite{Jamali}, Corollary 3.7]). Our first main result extends this statement to the more general class of locally generalized radical groups.
\medskip

\begin{theorem}
\label{main1}
Let $G$ be a locally generalized radical group such that $\gamma_s(G)/(\gamma_s(G) \cap Z_t(G))$ has finite rank $r$ for some $s,t$. Then the rank of $\gamma_{s+j}(G)/(\gamma_{s+j}(G) \cap Z_{t-j}(G))$ is finite and $(j,r,s)$-bounded for every $j$ with $0 \leq j \leq t$. In particular, the rank of $\gamma_{s+t}(G)$ is finite and $(r,s,t)$-bounded.
\end{theorem}
\medskip

Obviously, this is an extension of the result [\cite{KO2013}, Theorem A], which can be recovered from Theorem \ref{main1} as a particular case where $s=1$. \par 
A corresponding analogue of Baer’s theorem is also known for Cher\-nikov groups, namely that if $G/Z_s(G)$ is a Chernikov group, then $\gamma_{s+1}(G)$ is still a Chernikov group (see for instance [\cite{Robinson} Part 1, Theorem 4.21 and Corollary 2]). \par
Recall that a group $G$ is \emph{Chernikov} if it contains a radicable abelian normal subgroup $R$ (the finite residual) which is the direct product of a finite number $m(G)$ of groups of Prüfer type, such that the factor group $G/R$ is finite, of order $n(G)$ say. In general a group $G$ is called \emph{radicable} if the equation $x^n=a$ has a solution in $G$ for every positive integer $n$ and every $a \in G$. By a deep result obtained independently by Shunkov \cite{Shunkov} and Kegel and Wehrfritz \cite{Kegel}, Chernikov groups are precisely the locally finite groups satisfying the minimal condition on subgroups, that is, any non-empty set of subgroups possesses a minimal subgroup. Furthermore, if $G$ is a Chernikov group with $m=m(G)$ and $n=n(G)$, the pair $(m,n)$ is usually called the \emph{size} of $G$. Notice also that both the rank of $G$ and the number of primes involved in $G$ are bounded by $m(G) + n(G)$. Of course, a Chernikov group $G$ is finite if and only if $m(G)=0$, and it is a radicable abelian group if and only if $n(G)=1$. 
It is well-known that the class of Chernikov groups is closed with respect to subgroups, homomorphic images and extensions. In particular, if $G$ is a Chernikov group of size $(m,n)$, then each quotient of $G$ is a Chernikov group of size at most $(m,n)$.  \par Finally, we establish the following theorem.
\medskip

\begin{theorem}
\label{main2}
Let $m ,n$ be non-negative integers, with $n \neq 0$.
Let $G$ be a group such that $\gamma_s(G)/(\gamma_s(G) \cap Z_t(G))$ is a Chernikov group of size $(m,n)$ for some $s,t$. Then $\gamma_{s+j}(G)/(\gamma_{s+j}(G) \cap Z_{t-j}(G))$ is a Chernikov group of size bounded in terms of $j,m,n$ and $s$ for every $j$ with $0 \leq j \leq t$. In particular, $\gamma_{s+t}(G)$ is a Chernikov group of size bounded in terms of $m,n,s$ and $t$.
\end{theorem}

\section{Proofs}

In the present paper we make no attempts to write down explicit bounds for the rank and for the Chernikov size of $\gamma_{s+t}(G)$ in Theorem \ref{main1} and Theorem \ref{main2}, respectively. Furthermore, we use without explicit references the facts that if $r(G) = r$, then every
subgroup and every quotient of $G$ has rank at most $r$ and that if $G$ has a normal subgroup $N$ such that $r(N) = r_1$ and $r(G/N) = r_2$, then
$r(G) \leq r_1 + r_2$. 

The following two results of \cite{pavel2013} are essential to our proof of Theorem \ref{main1}.
\medskip

\begin{theorem}[\cite{pavel2013}, Theorem 1.1]
\label{2.8}
Let $G$ be a finite group such that $G/Z(G)$ has rank $r$. Then the rank of $G'$ is $r$-bounded.
\end{theorem}
\medskip

\begin{theorem}[\cite{pavel2013}, Theorem 1.2]
\label{T1}
Let $G$ be a locally generalized radical group such that $G/Z(G)$ has finite rank $r$. Then the rank of $G'$ is finite and $r$-bounded.
\end{theorem}
\medskip

The following results will be useful later on. Recall that if $G$ is a group, then $\gamma_\infty(G)$ denotes the last term of the lower central series of $G$.  
\medskip

\begin{lemma}[\cite{Conciseness}, Lemma 2.4]
\label{2.4}
Let $G$ be a finite metanilpotent group, and let $p$ be a prime. If $P$ is a Sylow $p$-subgroup of $\gamma_\infty(G)$
and $H$ is a Hall $p'$-subgroup of $G$, then $P=[P, H]$.
\end{lemma}
\medskip

Before starting the proofs, we recall a simple fact that will be used several times. If $G$ is a group, $N$ a normal subgroup of $G$ and $\overline{G}=G/N$, then
$\gamma_s(\overline{G}) / (\gamma_s(\overline{G}) \cap Z(\overline{G}))$ is a homomorphic image of the section $\gamma_s(G) / (\gamma_s(G) \cap Z(G))$. Consequently, the property that this section has finite rank $r$ is inherited by quotients. The same holds when such section is a Chernikov group of bounded size. \par
Therefore, as an immediate consequence of [\cite{Jamali}, Proposition 3.6], we obtain the following corollary.
\medskip

\begin{corollary}
\label{Jamali2}
Let $G$ be a finite nilpotent group
and let $s \geq 1$ be an integer such that $\gamma_s(G) / (\gamma_s(G) \cap Z(G))$ has rank $r$. Then $\gamma_{s+1}(G)$ has $(r,s)$-bounded rank.
\end{corollary}
\medskip

The next lemma shows that the hypothesis that $G$ is nilpotent in Corollary \ref{Jamali2} is superfluous.
\begin{lemma}
\label{L2}
Let $G$ be a finite group and let $s \geq 1$ be an integer such that $\gamma_s(G) / (\gamma_s(G) \cap Z(G))$ has rank $r$. Then $\gamma_{s+1}(G)$ has $(r,s)$-bounded rank.
\end{lemma}

\begin{proof}
Since $\gamma_s(G)Z(G) / Z(G)$ has rank $r$, Theorem \ref{2.8} yields that $(\gamma_s(G)Z(G))' = \gamma_s(G)'$ has $r$-bounded rank. Thus, passing to the quotient $G/\gamma_s(G)'$, we may assume that $\gamma_s(G)$ is abelian. In particular, $G$ is metanilpotent. \par
Let $P$ be a Sylow $p$-subgroup of $\gamma_\infty(G)$, where $p$ is a prime, and let $H$ be a Hall $p'$-subgroup of $G$. By Lemma \ref{2.4}, we have $P=[P,H]$. Moreover, since $P$ is normal in $G$, it follows that $P =  C_P(H) \times [P,H]$ (see [\cite{Gorenstein}, Theorem 2.3]), hence $C_P(H)$ is trivial. On the other hand, $P \cap Z(G) \leq C_P(H)$, and therefore $P$ embeds into $\gamma_s(G) / (\gamma_s(G) \cap Z(G))$. Consequently, $P$ has rank at most $r$. Since this holds for every Sylow subgroup of $\gamma_\infty(G)$, we deduce that $\gamma_\infty(G)$ itself has rank at most $r$.  \par
Finally, applying Corollary \ref{Jamali2} to the nilpotent quotient $G/\gamma_\infty(G)$, we conclude that $\gamma_{s+1}(G)/\gamma_\infty(G)$ has $(r,s)$-bounded rank. Hence, $\gamma_{s+1}(G)$ has $(r,s)$-bounded rank, as required. 
\end{proof}
\medskip

The key step in the proof of Theorem \ref{main1} is in the following proposition. It is a rank analogue of [\cite{AlcMor}, Proposition 2.3]. Note that quotients of a generalized radical group are generalized radical as well. Consequently, the class of locally generalized radical group is also closed under taking quotient. 
\medskip

\begin{proposition}
\label{T3}
Let $G$ be a locally generalized radical group
and let $s \geq 1$ be an integer such that $\gamma_s(G) / (\gamma_s(G) \cap Z(G))$ has finite rank $r$. Then the rank of $\gamma_{s+1}(G)$ is finite and $(r,s)$-bounded.
\end{proposition}

\begin{proof}
Since the group $\gamma_s(G)Z(G) / Z(G)$ has rank $r$, Theorem \ref{T1} implies that $(\gamma_s(G)Z(G))' = \gamma_s(G)'$ has $r$-bounded rank. Therefore passing to the quotient $G/\gamma_s(G)'$ we may assume that $\gamma_s(G)$ is abelian. \par Suppose first that $G$ is finitely generated. In this case $G$ is residually finite (see [\cite{Robinson} Part 2, Theorem 9.51]). Observe that 
$$\gamma_{s+1}(G)/(\gamma_{s+1}(G) \cap Z(G))$$ naturally embeds into $\gamma_s(G) / (\gamma_s(G) \cap Z(G))$. Therefore the quotient $\gamma_{s+1}(G) / (\gamma_{s+1}(G) \cap Z(G))$ has finite rank at most $r$. So it is sufficient to show that $\gamma_{s+1}(G) \cap Z(G)$ has  finite $(r,s)$-bounded rank. 
Let $E$ be a finitely generated subgroup of $\gamma_{s+1}(G) \cap Z(G)$. Since $E$ is abelian, 
\[ r(E) = \underset{p \in \mathbb{P}}{\mbox{max}} \ r_p(E).
\]
Fix a prime number $p$. Note that $E^p$ is normal (even central) in $G$ and put $\overline{G} = G/E^p$. We have that $\overline{E} = E/E^p$ is a finite $p$-elementary abelian subgroup of the residually finite group $\overline{G}$. Therefore
there exists a normal subgroup $\overline{N}$ of $\overline{G}$ such that $\overline{G}/\overline{N}$ is finite and $\overline{E} \cap \overline{N} = \{1\}$. It follows that 
$\overline{E}$ is isomorphic to a subgroup of $\gamma_{s+1}(\overline{G}/\overline{N})$.

Applying Lemma \ref{L2} to the finite group $\overline{G}/\overline{N}$, we obtain that the
rank of $\gamma_{s+1}(\overline{G}/\overline{N})$ is finite and $(r,s)$-bounded, and thus also $E/E^p$
has  finite $(r,s)$-bounded rank. Since this holds for every prime $p$, it follows that $\gamma_{s+1}(G) \cap Z(G)$ has  finite $(r,s)$-bounded rank and this proves the claim when $G$ is finitely generated. \par In other words, there exists an $(r,s)$-bounded number, say $R_0$, such that $r(\gamma_{s+1}(K)) \leq R_0$ whenever a group $K$ satisfies the
hypothesis of the theorem and is finitely generated. Suppose now that
our group $G$ is not necessarily finitely generated. If $r(\gamma_{s+1}(G)) \geq R_0 +1$, we can choose $y_1,\dots , y_{R_0+1} \in \gamma_{s+1}(G)$ such that the subgroup $\langle y_1,\dots , y_{R_0+1} \rangle$
cannot be generated by $R_0$ elements. We can also choose a finitely generated subgroup $K$ in $G$ such that $y_1,\dots , y_{R_0+1} \in \gamma_{s+1}(K)$. This yields a contradiction since we know that $r(\gamma_{s+1}(K)) \leq R_0$. The proof is now complete.
\end{proof}
\medskip

We are now ready to prove Theorem \ref{main1}.
\medskip

\begin{proof}[Proof of Theorem \ref{main1}]
We proceed by induction on $t$. Clearly, the statement holds if $t=0$. Assume that $t > 0$, and consider the group $\overline{G} = G/Z_{t-1}(G)$. 
Since $\overline{G}$ is still a locally generalized radical group, applying Proposition \ref{T3} to $\overline{G}$, it follows that 
$\gamma_{s+1}(\overline{G}) $ has finite and $(r,s)$-bounded rank. Now, by inductive hypothesis, we obtain that $\gamma_{s+j}(G)/(\gamma_{s+j}(G) \cap Z_{t-j}(G))$ has $(j,r,s)$-bounded rank, for every $j$ with $0 \leq j \leq t$.
\end{proof}
\medskip

The following proposition provides an analogue of Proposition \ref{T3} for the class of Chernikov groups of bounded size.
\medskip

\begin{proposition}
\label{main}
Let $m,n,s$ be non-negative integers, with $n,s \neq 0$. Let $G$ be a group such that $\gamma_s(G)/(\gamma_s(G) \cap Z(G))$ is a Chernikov group of size $(m,n)$. Then $\gamma_{s+1}(G)$ is a Chernikov group of size bounded in terms of $m,n$ and $s$.
\end{proposition}

\begin{proof}
Since $\gamma_s(G)Z(G) / Z(G)$ is a Chernikov group of size $(m,n)$, it follows from [\cite{DeFalco}, Lemma 4] that $(\gamma_s(G)Z(G))' = \gamma_s(G)'$ is a Chernikov group of size  bounded in terms of $m$ and $n$. Thus, we may replace $G$ by $G/\gamma_s(G)'$ and assume that $\gamma_s(G)$ is abelian. 
\par 
Let $a \in \gamma_s(G)$. There exists a positive integer $k$ such that $[a^k, g] = 1$ for every $g \in G$. As $\gamma_s(G)$ is abelian, we have $[a,g]^k = [a^k, g] = 1$ and therefore $\gamma_{s+1}(G)$ is periodic. \par
Assume first that $\gamma_{s+1}(G)$ is a $p$-group, for some prime $p$. The group $G$ obviously satisfies the assumptions of Proposition \ref{T3} with $r = m + n$, hence it follows that the rank of $\gamma_{s+1}(G)$ is finite and bounded by a function $u=u(m,n,s)$. It follows that the radicable part $R$ of $\gamma_{s+1}(G)$ is a direct product of at most $u$ Prüfer subgroups. In particular, $R$ is also normal in $G$. Let $D/(\gamma_s(G) \cap Z(G))$ be the radicable part of $\gamma_s(G)/(\gamma_s(G) \cap Z(G))$; then $\gamma_s(G)/D$ is finite of order $n$. For convenience, put $Z/R=\gamma_s(G/R) \cap Z(G/R)$ and let $g \in G$. Since $\gamma_s(G)$ is abelian, the map
\[ x(\gamma_s(G) \cap Z(G)) \in D/(\gamma_s(G) \cap Z(G)) \to [x,g] \in [D,g]\] induces an epimorphism from $D/(\gamma_s(G) \cap Z(G))$ to $[D,g]$. Thus $[D,g]$ is a radicable abelian subgroup of $\gamma_{s+1}(G)$, consequently $[D,g] \leq R$. Therefore $D \leq Z$ and so 
\[\frac{\gamma_s(G/R)}{\gamma_s(G/R) \cap Z(G/R) } \simeq \gamma_s(G)/Z
\] is finite of order at most $n$.
Applying [\cite{AlcMor}, 
Proposition 2.3] to the quotient $G/R$, we conclude that
$\gamma_{s+1}(G)/R$ has finite order bounded by a function $v=v(n,s)$, and so $\gamma_{s+1}(G)$ is a Chernikov group of size at most $(u, v)$. \par
Suppose now that $\gamma_{s+1}(G)$ is not necessarily a $p$-group. Since the quotient $\gamma_s(G)/(\gamma_s(G) \cap Z(G))$ is a Chernikov group of size $(m,n)$, the set of primes involved in this quotient has cardinality at most $m+n$.
%$ |\pi(\gamma_s(G)/\gamma_s(G) \cap Z(G))| \leq m+n$ 
Furthermore, since $\gamma_s(G)$ is abelian, 
for every $g \in G$, the subgroup $[\gamma_s(G),g]$ is a homomorphic image of $\gamma_s(G)/(\gamma_s(G) \cap Z(G))$ through the map
\[
a \in \gamma_s(G) \to [a,g] \in [\gamma_s(G),g].
\]
Hence every prime occurring in the torsion of $[\gamma_s(G),g]$ also occurs in the quotient $\gamma_s(G)/(\gamma_s(G) \cap Z(G))$. Moreover, since $\gamma_{s+1}(G)$ is abelian, it is generated by the subgroups of the form $[\gamma_s(G),g]$. Hence, also the primes involved in $\gamma_{s+1}(G)$ are contained in the set of the primes involved in $\gamma_s(G)/(\gamma_s(G) \cap Z(G))$ and their number is at most $m+n$. For each prime $p$, let $O_{p'}(\gamma_{s+1}(G))$ be the subgroup generated by all Sylow $q$-subgroups of $\gamma_{s+1}(G)$, with $q \neq p$. Therefore we apply the previous case to the quotient $G/O_{p'}(\gamma_{s+1}(G))$ and thus we obtain that each Sylow $p$-subgroup of $\gamma_{s+1}(G)$ is a Chernikov group of size at most $(u, v)$, and so $\gamma_{s+1}(G)$ itself is a Chernikov group of size at most $(u(m+n), v^{m+n})$.
\end{proof}
\medskip

The proof of Theorem \ref{main2} follows the same lines as that of Theorem \ref{main1}, with Proposition \ref{main} taking the place of Proposition \ref{T3}, and is therefore omitted.

\end{document}